\documentclass[11pt]{article}
\usepackage{amsmath}
\usepackage{amsfonts}
\usepackage{setspace}
\usepackage{enumerate}
\usepackage{epsfig}
\usepackage{subfigure}
\usepackage{rotating}
\usepackage{color}
\usepackage{multirow}
\numberwithin{equation}{section}
\newtheorem{defn}{Definition}[section]
\newtheorem{thm}{Theorem}[section]
\newtheorem{prop}{Proposition}[section]
\newtheorem{rem}{Remark}[section]
\newtheorem{cor}{Corollary}[section]
\newenvironment{proof}[1][Proof]{\begin{trivlist}
\item[\hskip \labelsep {\bfseries #1}]}{\end{trivlist}}

\begin{document}

\begin{center}

 \section*{Lidstone Fractal Interpolation and Error Analysis}
\end{center}

\begin{center}  G. P. Kapoor \\* Department of Mathematics and Statistics \\* Indian Institute of Technology Kanpur, Kanpur 208016, India \\* email: gp@iitk.ac.in \end{center}

\begin{center} and \end{center}

\begin{center} M. Sahoo \\* School of Applied Sciences \\* KIIT University, Bhubaneswar 751024, Odisha, India \\* email: msahoo@gmail.com \end{center}

\doublespacing

\begin{abstract}
\noindent In the present paper, the notion of Lidstone Fractal Interpolation Function ($Lidstone \ FIF$) is introduced to interpolate and approximate data generating functions that arise from real life objects and outcomes of several scientific experiments. A Lidstone FIF extends the classical Lidstone Interpolation Function which is generally found not to be satisfactory in interpolation and approximation of such functions. For a data $\{(x_n,y_{n,2k}); n=0,1,\ldots,N \ \text{and} \ k=0,1,\ldots,p\}$  with $N,p\in\mathbb{N}$, the existence of Lidstone FIF is proved in the present work and a computational method for its construction is developed.  The constructed Lidstone FIF is a $C^{2p}[x_0,x_N]$ fractal function $\ell_\alpha$ satisfying $\ell_\alpha^{(2k)}(x_n)=y_{n,2k}$, $n=0,1,\ldots,N$,\ $k=0,1,\ldots,p$. Our error estimates establish that the order of $L^\infty$-error in approximation of a data generating function in $C^{2p}[x_0,x_N]$  by Lidstone FIF is of the order $N^{-2p}$, while $L^\infty$-error in approximation of $2k$-order derivative of the data generating function by corresponding order derivative of Lidstone FIF is of the order $N^{-(2p-2k)}$. The results found in the present work are illustrated for computational constructions of a Lidstone FIF and its derivatives with an example of a data generating function.
\end{abstract}

{\bf Key Words :} Lidstone, Interpolation, Approximation, Fractal, FIF,  IFS, Error Estimate

\newpage

\section{Introduction}
The Lidstone Polynomials \cite{l2} were introduced to offer an approximation of sufficient number of times continuously differentiable functions at two points instead of their one point approximation given by Taylor Polynomials. A classical Lidstone Interpolation Function on the interval $[x_0,x_N]$ is comprised of a Lidstone Polynomial for each subinterval arising from a partition of the interval. It approximates a function in $C^{2p}[x_0,x_N]$, the class of $2p$ times  continuously differentiable functions, such that the Lidstone Interpolation Function and all its even order derivatives up to the order $2p$  agree with the given function and its corresponding derivatives at finitely many abscissas of data points in the interval including its end points. Such an interpolation is widely used in the boundary value problem consisting of the $2p^{th}$-order ordinary differential equation $y^{(2p)}= - f(x,y,y', \ldots,y^{(2p-1)})$ with Lidstone Boundary Conditions $y^{(2i)}(x_0)
=\alpha_i$, $y^{(2i)}(x_N)=\beta_i, \  i=0,1,...,p$ \cite{a6,a3,b20,t2}. However the classical Lidstone Interpolation is generally not suited for approximation of fractal functions \cite{b3}. Such functions quite often arise in various science and engineering problems \cite{a10,c6,i2,k2}, medicine\cite{i1}, economics \cite{m8,w5}, arts\cite{b17}, and music\cite{m1}. In the present paper we introduce the notion of Lidstone Fractal Interpolation to approximate such fractal functions with respect to $L^\infty$-norm.

The organization of the paper is as follows. In  section
\ref{basica}, we review some basic definitions and results on
classical  Lidstone Interpolation that are used  in the later sections of the paper. In Section \ref{consLFIF}, for an interpolation data in the Euclidean plane $\mathbb{R}^2$, the notion of Lidstone FIF is introduced, its existence is established and a computational method of construction of Lidstone FIF is developed. The convergence of Lidstone FIF and its even order derivatives to the data generating function and its corresponding derivatives are studied in Section \ref{convLFIF}. To this end, first the continuous dependence of Lidstone FIF $\ell_\alpha$ and its even order derivatives on the parameter  $\alpha$ is established. This is followed by proving that, for the classical Lidstone Interpolation Function $\phi =$ $\phi^{(0)}$ and its even order derivatives $\phi^{(2k)}$  and Lidstone FIF $\ell_\alpha =$ $\ell_\alpha^{(0)}$ and its corresponding derivatives $\ell_\alpha^{(2k)}$, the $L^\infty$-norm $ \|\phi^{(2k)}-\ell_\alpha^{(2k)}\|_\infty \rightarrow 0$, $k=0,1,2,\ldots,p$, as norm of the  partition of the interval $[x_0,x_N]$ consisting abscissas of data points tends to zero. Using these results, it is found in this section that $L^\infty$-error in approximation of data generating function in $C^{2p}[x_0,x_N]$  by Lidstone FIF is of the order $N^{-2p}$, while $L^\infty$-error in approximation of $2k$-order derivative of data generating function in $C^{2p}[x_0,x_N]$  by corresponding order derivative of Lidstone FIF is of the order $N^{-(2p-2k)}$. Finally, the results of Section \ref{consLFIF} on computational constructions of a Lidstone FIF and its derivatives are illustrated with an example in Section \ref{examplesc} .

\section{Classical Lidstone and Piecewise-Lidstone \\* Interpolation: Basic Concepts} \label{basica}

\label{lidstone}
It is known \cite{w6} that a function $g: \mathbb{R}
\rightarrow \mathbb{R}$ possessing a sufficient number of derivatives can be expanded as Lidstone Series \begin{equation*}
g(x)=\sum_{l=0}^p\left[g^{(2l)}(0)\Lambda_l(1-x)+g^{(2l)}(1)\Lambda_l(x)\right]+R_{p+1}(g,x),
\end{equation*} where, $\mathbb{R}$ is the real line, $\Lambda_{l}(x)$ are the \textit{Lidstone Polynomials }defined by means of recursive relations:
\begin{equation}\label{eqnc1}\left.
\begin{array}{l}
\Lambda_{0}(x)=x\\
\Lambda''_{l}(x)=\Lambda_{l-1}(x), ~~l\geq 1\\
\Lambda_{l}(0)=\Lambda_{l}(1)= 0, ~~l\geq 1
\end{array}\right\}
\end{equation}
and $R_{p+1}(g,x)=\int_0^1G_{p+1}(x,t)g^{(2p+2)}(t)~dt$, where
\begin{equation*}\label{eqnc3}
G_1(x,t)=\left\{ \begin{array}{cc}(x-1)t,& t\leq x\\ (t-1)x, &
x\leq t \end{array} \right.
\end{equation*}

\begin{equation*}\label{eqnc4}
G_p(x,t)=\int_0^1g_1(x,s)G_{p-1}(s,t)~ds,~~~~p\geq 2.
\end{equation*}

\noindent The following proposition gives an explicit form
of Lidstone Polynomial and its bound in the interval $[0,1]$: \newpage

\begin{prop}\cite{a2}\label{propc2}
The Lidstone Polynomial $\Lambda_{l}(x)$, in the interval $[0,1]$,
can be expressed as \begin{equation*}\label{eqnc5}
\Lambda_{l}(x)=(-1)^{l}\frac{2}{\pi^{2l+1}}\sum_{n=1}^{\infty}\frac{(-1)^{n+1}}{n^{2l+1}}\sin
n\pi x,\ \ \ \ l\geq 1
\end{equation*}
so that,
\begin{equation}\label{eqnc6}
|\Lambda_{l}(x)|\leq \frac{1}{3\pi^{2l-1}};\ \ \ \ 0\leq x \leq 1.
\end{equation}
\end{prop}
\noindent The truncated Lidstone series
\begin{equation*}
\sum_{l=0}^p\left[\alpha_{2l}\Lambda_l(1-x)+\beta_{2l}\Lambda_l(x)\right]
\end{equation*}
for $\alpha_{2l},\beta_{2l}\in \mathbb{R}$, $l=0,1,\ldots,p$, is
known as  \textit{Lidstone Interpolating Polynomial} in the
interval $[0,1]$. To extend Lidstone Interpolating Polynomial to
the interval $[x_0,x_N]$, we have

\begin{defn}\label{lip}
A \ real \ polynomial \ $q(x)$ of degree $2p+1$, satisfying  the
Lidstone Conditions \ $q^{(2l)}(x_0)=y_{0,2l}$,
$q^{(2l)}(x_N)=y_{N,2l}$, $0\leq l\leq p$, \ $y_{0,2l}, \
y_{N,2l} \in \mathbb{R}$, is known as a \textit{Lidstone
Interpolating Polynomial} in $[x_0,x_N]$.
\end{defn}
For given $y_{0,2l}, \ y_{N,2l}$, $l=0,1,\ldots,p$, a representation
of the Lidstone Interpolating Polynomial $q(x)$ is given by the
following theorem:
\begin{thm}\cite{a2}\label{thmc1}
The Lidstone Interpolating Polynomial $q(x)$ can be expressed as
\begin{equation*}\label{eqnc8}
q(x)=\sum_{l=0}^p\left[y_{0,2l}\Lambda_{l}\left(\frac{x_N-x}{x_N-x_0}\right)
+y_{N,2l}\Lambda_{l}\left(\frac{x-x_0}{x_N-x_0}\right)\right](x_N-x_0)^{2l}.
\end{equation*}
\end{thm}

It is apparent that Lidstone Interpolating Polynomial
interpolates a data given only at two points. To interpolate a
data given at more than two points, Lidstone Interpolation Functions defined below are employed.
Let \~{S}$^{p}[x_0,x_N]$, where $-\infty<x_0<x_N<\infty$ and $p$ is a positive integer, be the set of all real-valued functions $g(x)$ that satisfy the following conditions:\\
(i) The function $g$ is $p$ times continuously differentiable on $[x_0,x_N]$.\\
(ii) There exists a partition  $\Delta:x_0<x_1<\cdots < x_N$ such that on each open subinterval $(x_{n-1},x_n)$, the $p$\textsuperscript{th}-derivative $g^{(p)}$ is continuously differentiable.\\

(iii) The sup-norm of $g^{(p+1)}$ is finite, i.e.,
\begin{equation*}\label{eqnc9}
\|g^{(p+1)}\|_\infty=\max_{1\leq i \leq N} \sup_{x\in
(x_{n-1},x_n)}|g^{(p)}(x)| < \infty.
\end{equation*}
For a fixed uniform partition $\Delta:x_0<x_1<\cdots < x_N$ of the
interval $[x_0,x_N]$, define
\begin{eqnarray*}\label{eqnb8a}
L^\Delta_{p+1}&=&\left\{\varphi~|~\varphi\in C[x_0,x_N] \text{ and
}
\phi \text{ is a polynomial } \text{ of degree}\right.\nonumber \\
& &~~~~~~~~\left.\text{at most} \text{ 2p+1 in each sub
interval }[x_{n-1},x_n]\right\}
\end{eqnarray*}
where, $C[x_0,x_N]$ is the set of all continuous functions in $[x_0,x_N]$. It is easily seen that $L^\Delta_{p+1}$ is a linear space with usual point-wise addition and scalar multiplication of functions.
\begin{defn}
For a given function $g\in C^{2p}[x_0,x_N]$,  \ $L_{p+1}^\Delta g(x)$
is called the $L_{p+1}^\Delta$-interpolate (also called {\it
Lidstone Interpolation Function})\quad of \ $g(x)$, if \quad $L_{p+1}^\Delta g\in L_{p+1}^\Delta$ and $\frac{d^{2k}}{dx^{2k}}L_{p+1}^\Delta
g(x_n)=g^{(2k)}(x_n)$; $0\leq k\leq p$, $0\leq n\leq N$.
\end{defn} \noindent
In view of Theorem \ref{thmc1}, $L_{p+1}^\Delta$-interpolate for a
function $g\in C^{2p}[x_0,x_N]$ uniquely exists and, for $x \in
[x_{n-1},x_n]$,
\begin{equation*}\label{eqnc10}
L_{p+1}^\Delta g(x)=
\sum_{l=0}^p\left[g^{(2l)}(x_{n-1})\Lambda_{l}\left(\frac{x_n-x}{x_n-x_{n-1}}\right)+g^{(2l)}(x_n)\Lambda_{l}
\left(\frac{x-x_{n-1}}{x_n-x_{n-1}}\right)\right](x_n-x_{n-1})^{2l}.
\end{equation*}
Thus,
\begin{equation*}\label{eqnc11}
L_{p+1}^\Delta g(x)=\sum_{n=0}^N
\sum_{l=0}^pr_{p,n,l}(x)~g^{(2l)}(x_n)
\end{equation*}
where,
\begin{equation*}\label{eqnc12}
r_{p,n,l}(x)= \left\{\begin{array}{ccc}
           \Lambda_l\left(\frac{x-x_{n-1}}{x_n-x_{n-1}}\right)(x_n-x_{n-1})^{2l}, & x_{n-1}\leq x\leq x_n, & 1\leq n\leq N\\
           \Lambda_l\left(\frac{x_{n+1}-x}{x_{n+1}-x_n}\right)(x_{n+1}-x_n)^{2l}, & x_n\leq x\leq x_{n+1}, & 0\leq n\leq N-1\\
           0,& \text{otherwise}.& \\
       \end{array}
\right.
\end{equation*}
It is easily verified that the set of functions $\{r_{p,n,l};n=0,1,\ldots N, l=0,1,\ldots,p\}$ forms a basis of the linear space $L_{p+1}^\Delta$. \newpage

\noindent The error in approximation of a function belonging to \~{S}$^{2p}[x_0,x_N]$ and its even order derivatives by an $L^\Delta _{p+1}$-interpolate and its corresponding derivatives is given by the following theorem:
\begin{thm}\cite{a2}\label{thmc3}
Let $g \in$  \~{S} $^{2p}[x_0,x_N]$. Then
\begin{equation}\label{eqnc16}
\left\|\frac{d^k}{dx^k}(g-L^\Delta _{p+1}g)\right\|_\infty \leq
2d_{2p,k} \|g^{(2p)}\|_\infty \|\Delta\|^{2p-k},~~ 0\leq k \leq 2p
\end{equation}
where, $\|\Delta\|=\max\{~x_n-x_{n-1}~; ~~1\leq n\leq N\}$ and
\begin{equation}\label{eqnc15}
d_{2p,k}= \left\{\begin{array}{lll}
             \frac{(-1)^{p-i}E_{2p-2i}}{2^{2p-2i}(2p-2i)!}, &k=2i, & 0\leq i\leq p\\
             (-1)^{p-i+1}\frac{2(2^{2p-2i}-1)}{(2p-2i)!}B_{2p-2i}, & k=2i+1, & 0\leq i\leq p-1\\
             2,& k=2p+1 & \\
       \end{array}
\right.
\end{equation}
$E_{2p}$ and $B_{2p}$ being $2p^{th}$ {\it Euler} and
{\it Bernoulli numbers} respectively. \end{thm}

\section{Construction of Lidstone FIF } \label{consLFIF}

In this section, we introduce Lidstone Fractal Interpolation Function, prove the existence and give a computational method for its construction. We begin with the definition of Lidstone FIF for a
given data:

Let $x_0<x_1<\cdots<x_N$ be a partition of the interval
$[x_0,x_N]$ and $\Theta_{2p} =
\{\alpha~|~\alpha=(\alpha_1,\alpha_2,\ldots,\alpha_N)\in\mathbb{R}^N,
|\alpha_n|<a_n^{2p}\}$, where $a_n=\frac{x_n-x_{n-1}}{x_N-x_0}$.
For $\alpha \in \Theta_{2p}$, consider the Iterated Function System  (\textit {IFS}) $\{\mathbb{R}^2; w_n(x,y)=(L_n(x), F_n(x,y)), n=1,2,\ldots,N\}$, where
$L_n:\mathbb{R} \rightarrow \mathbb{R}$ and $F_n:\mathbb{R}^2
\rightarrow \mathbb{R}$, for $n=1,2,\ldots,N$, are given by,
\begin{equation}\label{eqnc17}\left.
\begin{array}{ccc}
L_n(x)&=&a_nx+b_n\\
F_n(x,y)&=&\alpha_ny+q_n(x)
\end{array}\right\}
\end{equation}
with $b_n=\frac{x_Nx_{n-1}-x_0x_n}{x_N-x_0}$ and $q_n\in
C^{2p}(\mathbb{R)}$. For a given set of real numbers $y_n$,
$n=0,1,\ldots,N$, if the IFS satisfies the join up conditions
\begin{equation*}
\left.\begin{array}{l}
w_n(x_0,y_0)=(x_{n-1},y_{n-1})\\
w_n(x_N,y_N)=(x_n,y_n)
\end{array}\right\}
\end{equation*}

for $n=1,2,\ldots,N$, then
the attractor of the IFS is the graph of a $C^{2p}$-fractal
function $\ell_\alpha:[x_0,x_N] \rightarrow \mathbb{R}$ such that
$\ell_\alpha(x_n)=y_n$.

\begin{defn}
The \ $C^{2p}$-fractal  \ function \ $\ell_\alpha$, \ generated \ by \ the \ IFS
$\{\mathbb{R}^2; w_n(x,y)= (a_nx+b_n, \alpha_ny+q_n(x)),
n=1,2,\ldots,N\}$, is called \textit{Lidstone Fractal
Interpolation Function} (\textit{Lidstone FIF}) for a data
$\{(x_n,y_{n,2k});n=0,1,\ldots,N \text{ and } k=0,1,\ldots,p\}$, if each $q_n$ is a Lidstone Interpolating Polynomial of degree at most $2p+1$ in $[x_0,x_N]$ and $\ell_\alpha^{(2k)}(x_n)=y_{n,2k}$. \newline \\ The tuple $\alpha=(\alpha_1,\alpha_2,\ldots,\alpha_N)\in\mathbb{R}^N $ is called scaling factor of the Lidstone FIF $\ell_\alpha$ \end{defn}

The existence and the method of construction  of Lidstone FIF is
given by the following theorem:

\begin{thm}\label{thmc4}
For  given data $\{(x_n,y_{n,2k});n=0,1,\ldots,N \text{ and }
k=0,1,\ldots,p\}$, with abscissas as partition points of $[x_0, x_N]$ and $\alpha \in \Theta_{2p}$, the Lidstone FIF $\ell_\alpha$ uniquely exists. \end{thm}
\begin{proof}
Consider IFS $\{\mathbb{R}^2; w_n(x,y)=(L_n(x), F_n(x,y)),
n=1,2,\ldots,N\}$, where $L_n:\mathbb{R} \rightarrow  \mathbb{R}$
, $F_n:\mathbb{R}^2 \rightarrow \mathbb{R}$ are given by (\ref{eqnc17}) and the polynomials $q_n$ of degree at most $2p+1$ are chosen as follows: \newline

\noindent Let the maps $F_{n,k}(x,y):\mathbb{R}^2
\rightarrow \mathbb{R}$, $k=0,1,\ldots,p$, be defined by
\begin{equation*}\label{eqnc18}\left.
\begin{array}{lll}
F_{n,0}(x,y)&=& F_n(x,y)\\
F_{n,2k}(x,y)&=&\frac{\alpha_ny+q_n^{(2k)}(x)}{a_n^{2k}},
k=1,2,\ldots,p,
\end{array}\right\}.
\end{equation*}
The  polynomials $q_n(x)$ are now chosen to satisfy the join up conditions
\\ \newline $F_{n,2k}(x_{0},y_{0,2k})$ $=
y_{n-1,2k}$ and $F_{n,2k}(x_{N},y_{N,2k})$ $=y_{n,2k}$ for
$k=0,1,\ldots,p$ \\ \newline so that
\begin{equation*}\label{eqnc19}
\frac{\alpha_ny_{0,2k}+q_n^{(2k)}(x_0)}{a_n^{2k}}=y_{n-1,2k}~~~~\text{and}~~~~
\frac{\alpha_ny_{N,2k}+q_n^{(2k)}(x_N)}{a_n^{2k}}=y_{n,2k}
\end{equation*}
for $k=1,2,\ldots, p$. Consequently, the chosen polynomials satisfy
\begin{equation*}\label{eqnc20}\left.
\begin{array}{lll}
q_n^{(2k)}(x_0)=a_n^{2k}y_{n-1,2k}-\alpha_ny_{0,2k}\\
q_n^{(2k)}(x_N)=a_n^{2k}y_{n,2k}-\alpha_ny_{N,2k}
\end{array}\right\},
\end{equation*}
for $k=1,2,\ldots, p$.

\noindent Equivalently, the polynomials $q_n$ are Lidstone Interpolating Polynomials of degree at most $2p+1$ in $[x_0,x_N]$ given by
(\textit{c.f. Theorem \ref{thmc1}})
\begin{equation}\label{eqnc21}
q_n(x)=\sum_{l=0}^p\left[q_n^{(2l)}(x_0)
\Lambda_l\left(\frac{x_N-x}{x_N-x_0}\right)+q_n^{(2l)}(x_N)
\Lambda_l\left(\frac{x-x_0}{x_N-x_0}\right)\right](x_N-x_0)^{2l}.
\end{equation} for the two point data $(x_0 ,\ a_n^{2k}y_{n-1,2k}-\alpha_ny_{0,2k})$, $(x_N ,a_n^{2k}y_{n,2k}-\alpha_ny_{N,2k})$. \newline

Now, IFS
$\{\mathbb{R}^2;w_n(x,y)=(L_n(x),F_n(x,y)):n=1,2,\ldots,N\}$,
determines a unique FIF $\ell_\alpha\in C^{2p}[x_0,x_N]$, which is
the fixed point of Read-Bajraktarevi\'{c} Operator $T_\alpha:\mathcal{F} \rightarrow \mathcal{F}$ given by
\begin{equation}\label{eqnc24}
T_\alpha\phi(x)=F_n(L_n^{-1}(x), \phi\circ L_n^{-1}(x));~~~x\in
[x_{n-1},x_n]
\end{equation}
where,
\begin{equation*}
\mathcal{F}=\left\{\phi:\mathbb{R}\rightarrow \mathbb{R}~~|~~\phi
\text{ is continuous and }
\phi(x_0)=y_{0,0},\phi(x_N)=y_{N,0}\right\}.
\end{equation*}
Thus, for $x\in [x_{n-1},x_n]$,
$\ell_\alpha(x)=\alpha_n \ \ell_\alpha\circ L_n^{-1}(x)+q_n\circ
L_n^{-1}(x)$.\newline

For establishing that $\ell_\alpha$ is the Lidstone FIF for the
given data, it remains to be shown that
$\ell_\alpha^{(2k)}(x_n)=y_{n,2k}$, for $n=0,1,\ldots,N$ and
$k=1,2,\ldots,p$. Since, by construction, for $n=1,2,\ldots,N$,
$F_{n,2k}(x_0,y_{0,2k})=y_{n-1,2k}$ and
$F_{n,2k}(x_N,y_{N,2k})=y_{n,2k}$, it follows that, the IFS
$\{\mathbb{R}^2$;$(L_n(x),F_{n,2k}(x,y))$,\ $n =1,2,\ldots,N\}$
determines an FIF $\ell_\alpha^{(2k)}$. In fact, repeated
applications of a result of Barnsley and Harrington (\cite{b9}, Theorem 1), show that $\ell_\alpha^{(2k)}$ is the $2k^{th}$
derivative of FIF $\ell_\alpha$. Further, since
$\ell_\alpha^{(2k)}$ is the fixed point of Read-Bajraktarevi\'{c}
operator $T_\alpha^{(2k)}:\mathcal{F}^{(2k)} \rightarrow
\mathcal{F}^{(2k)}$ given by
\begin{equation}\label{eqnc30}
T_\alpha^{(2k)}\phi(x)=F_{n,2k}(L_n^{-1}(x), \phi\circ
L_n^{-1}(x));~~~x\in [x_{n-1},x_n]
\end{equation}
where,
\begin{equation*}
\mathcal{F}^{(2k)}=\left\{\phi:\mathbb{R}\rightarrow
\mathbb{R}~~|~~\phi \text{ is continuous and }
\phi(x_0)=y_{0,2k},\phi(x_N)=y_{N,2k}\right\},
\end{equation*}
it follows that $\ell_\alpha^{(2k)}(x)=F_{n,2k}(L_n^{-1}(x),
\ell_\alpha^{(2k)}\circ L_n^{-1}(x))$. Since
$\ell_\alpha^{(2k)}\in\mathcal{F}^{(2k)}$,
$\ell_\alpha^{(2k)}(x_0)=y_{0,2k}$ and
$\ell_\alpha^{(2k)}(x_N)=y_{N,2k}$.
Further, for $n=1,2,...,N-1$,
\begin{eqnarray*}
\ell_\alpha^{(2k)}(x_n)&=& F_{n,2k}(L_n^{-1}(x_n),\ell_\alpha^{(2k)}\circ L_n^{-1}(x_n))\\
&=&F_{n,2k}(x_N,\ell_\alpha^{(2k)}(x_N))\\
&=&F_{n,2k}(x_N,y_{N,2k})\\
&=&y_{n,2k}.
\end{eqnarray*}
Thus, $\ell_\alpha$ is the Lidstone FIF for the given data.
\end{proof}

\begin{rem}\label{remc2}
It follows from the proof of Theorem \ref{thmc4} that, for $x\in [x_{n-1},x_n],$ Lidstone FIF $\ell_\alpha$ is given by
\begin{equation}\label{eqnc25}
\ell_\alpha(x)=\alpha_n \ \ell_\alpha\circ L_n^{-1}(x)+q_n\circ
L_n^{-1}(x).
\end{equation}
Consequently,  $\ell_\alpha\in C^{2p}[x_0,x_N]$ and is a polynomial of degree at most $2p+1$ the interval $[x_{n-1},x_n]$, if $\alpha_n=0$. Thus, if
$\alpha=(\alpha_1,\alpha_2,\ldots,\alpha_N)=0$, the Lidstone FIF $\ell_\alpha\in L_{p+1}^\Delta$ and, for $x\in [x_{n-1},x_n]$, is given by
\begin{equation*}\label{eqnc25a}
\ell_\alpha(x)=\sum_{l=0}^p\left[y_{n-1,2l}\Lambda_{l}\left(\frac{x_n-x}{x_n-x_{n-1}}\right)+y_{n,2l}\Lambda_{l}
\left(\frac{x-x_{n-1}}{x_n-x_{n-1}}\right)\right](x_n-x_{n-1})^{2l}.
\end{equation*}
\end{rem}

\section{Convergence of Lidstone FIF and its Derivatives}\label{convLFIF}
The convergence of Lidstone FIF and its even order derivatives to the data generating function and its corresponding derivatives are studied in this section. To this end, the continuous dependence of Lidstone FIF $\ell_\alpha$ and its even order derivatives on the parameter  $\alpha$ is proved first.This is followed by proving that, for the classical Lidstone Interpolation Function $\phi =$ $\phi^{(0)}$ and its even order derivatives $\phi^{(2k)}$  and Lidstone FIF $\ell_\alpha =$ $\ell_\alpha^{(0)}$ and its corresponding derivatives $\ell_\alpha^{(2k)}$, the $L^\infty$-norm $ \|\phi^{(2k)}-\ell_\alpha^{(2k)}\|_\infty \rightarrow 0$, $k=0,1,2,\ldots,p$, as norm of the  partition of the interval $[x_0,x_N]$ consisting abscissas of data points tends to zero.  Using these results, it is found in this section that $L^\infty$-error in approximation of data generating function in $C^{2p}[x_0,x_N]$  by Lidstone FIF is of the order $N^{-2p}$, while $L^\infty$-error in approximation of $2k$-order derivative of data generating function in $C^{2p}[x_0,x_N]$  by corresponding order derivative of Lidstone FIF is of the order $N^{-(2p-2k)}$.

\ \ \ \

Let data $\{(x_n,y_{n,2k});n=0,1,\ldots,N \text{ and }
k=0,1,\ldots,p\}$ be generated by a function $g(x)\in
C^{2p}[x_0,x_N]$ (i.e., $g^{(2k)}(x_n)=y_{n,2k}$) \ and \ $\alpha \in \Theta_{2p}$. In the sequel, we use the following notations:
$|\alpha|_\infty=\max_n\{|\alpha_n|\}$,
$\rho=\max_k\{|y_{0,2k}|,|y_{N,2k}|\}$ and
$\|\ell_\alpha\|_\infty=\sup\{|\ell_\alpha(x)|~|~x\in[x_0,x_N]\}$ and
denote the polynomial $q_n(x)$ by $q_n(\alpha_n,x)$, to emphasize that it depends on the parameter $\alpha_n$, as observed in the proof of Theorem \ref{thmc4}.

We first prove  continuous dependence of Lidstone FIF $\ell_\alpha$ and its even order derivatives on the parameter  $\alpha$ with the help of following proposition:
\begin{prop}\label{propc4}
Let $q_n(\alpha_n,x)$, $n=1,2,\ldots,N$, be the polynomials
constructed in Theorem \ref{thmc4}. Then, for $k=0,1,2,\ldots,p$ and $x\in [x_0,x_N]$,
\begin{equation}\label{eqnc33}
\left|\frac{\partial^{2k+1}}{\partial\alpha_n
\partial x^{2k}}q_n(\alpha_n,x)\right|\leq \frac{2\rho\pi}{3}\sum_{l=0}^{p-k}\left(\frac{x_N-x_0}{\pi}\right)^{2l}
\end{equation}\end{prop}
\begin{proof}
The representation of $q_n(\alpha_n,x)$ as given by Equation (\ref{eqnc21}) is
\begin{equation}\label{eqnc34}
q_n(\alpha_n,x)=\sum_{l=0}^p\left[q_n^{(2l)}(x_0)
\Lambda_l\left(\frac{x_N-x}{x_N-x_0}\right)+q_n^{(2l)}(x_N)
\Lambda_l\left(\frac{x-x_0}{x_N-x_0}\right)\right](x_N-x_0)^{2l}
\end{equation}
where, $q_n^{(2l)}(x_0)=a_n^{2l}y_{n-1,2l}-\alpha_ny_{0,2l}$ and
$q_n^{(2l)}(x_N)=a_n^{2l}y_{n,2l}-\alpha_ny_{N,2l}$.
Differentiating both sides of (\ref{eqnc34}) with respect to
$\alpha_n$
\begin{equation}\label{eqnc35}
\frac{\partial}{\partial\alpha_n}q_n(\alpha_n,x)=
(-1)\sum_{l=0}^p\left[y_{0,2l}
\Lambda_l\left(\frac{x_N-x}{x_N-x_0}\right)+y_{N,2l}
\Lambda_l\left(\frac{x-x_0}{x_N-x_0}\right)\right](x_N-x_0)^{2l}.
\end{equation}
By Equation (\ref{eqnc35}) and Inequality (\ref{eqnc6}),
\begin{equation}\label{eqnc36}
\left|\frac{\partial}{\partial\alpha_n}q_n(\alpha_n,x)\right|\leq
\frac{2\rho\pi}{3}\sum_{l=0}^p\left(\frac{x_N-x_0}{\pi}\right)^{2l}
\end{equation}
for $x\in [x_0,x_N]$. Further, the properties of Lidstone
Polynomials (\textit{c.f. Section \ref{lidstone}}) and
(\ref{eqnc34}) give that, for $1\leq k\leq p$,
\begin{equation*}\label{eqnc37}
\frac{\partial^{2k}}{\partial
x^{2k}}q_n(\alpha_n,x)=\sum_{l=0}^{p-k}\left[q_n^{(2l+2k)}(x_0)
\Lambda_l\left(\frac{x_N-x}{x_N-x_0}\right)+q_n^{(2l+2k)}(x_N)
\Lambda_l\left(\frac{x-x_0}{x_N-x_0}\right)\right](x_N-x_0)^{2l}.
\end{equation*}
Thus, following the method of derivation of Inequality
(\ref{eqnc36}),
\begin{equation}\label{eqnc38}
\left|\frac{\partial^{2k+1}}{\partial\alpha_n
\partial x^{2k}}q_n(\alpha_n,x)\right| \leq
\frac{2\rho\pi}{3}\sum_{l=0}^{p-k}\left(\frac{x_N-x_0}{\pi}\right)^{2l}
\end{equation}
for all $x\in [x_0,x_N]$ and $1\leq k \leq p$. Inequality
(\ref{eqnc33}) follows from (\ref{eqnc36}) and (\ref{eqnc38}).
\end{proof}

The continuous dependence of Lidstone FIF $\ell_{\alpha}$ on the scaling factor $\alpha$ is deduced by the following theorem:
\begin{thm}\label{propc6}
Let $\ell_{\alpha'}$ and $\ell_{\alpha''}$ be the Lidstone $\textit FIFs$
with respect to scaling factors \ $\alpha', \alpha'' \in \Theta_{2p}$ respectively, for the data $\{(x_n,y_{n,2k});n=0,1,\ldots,N \text{and} \ k=0,1,\ldots,p\}$.
Then,
\begin{equation}\label{eqnc43}
\|\ell_{\alpha'}-\ell_{\alpha''}\|_\infty \leq
\frac{|\alpha'-\alpha''|_\infty}{1-|\alpha'|_\infty}
(\|\ell_{\alpha''}\|_\infty +M_{0,2p})
\end{equation}
where,
$M_{0,2p}=\frac{2\rho\pi}{3}\sum_{l=0}^p\left(\frac{x_N-x_0}{\pi}\right)^{2l}$.
\end{thm}
\begin{proof}
For $x\in [x_{n-1},x_n]$, it follows by (\ref{eqnc25}) and the mean value theorem that
\begin{eqnarray*}
|\ell_{\alpha'}(x)-\ell_{\alpha''}(x)|
&\leq&|\alpha_n'|~|\ell_{\alpha'}(L_n^{-1}(x))-\ell_{\alpha''}(L_n^{-1}(x))|\\
& &+|\alpha_n'-\alpha_n''|~|\ell_{\alpha''}(L_n^{-1}(x))|\\
&\leq&|\alpha_n'|~|\ell_{\alpha'}(L_n^{-1}(x))-\ell_{\alpha''}(L_n^{-1}(x))|\\
& &+|\alpha_n'-\alpha_n''|~|\ell_{\alpha''}(L_n^{-1}(x))|\\
&
&+|\alpha_n'-\alpha_n''|\left|\frac{\partial}{\partial\alpha_n}q_n(\xi_n,L_n^{-1}(x))\right|
\end{eqnarray*}
for some $\xi_n$ lying between $\alpha_n'$ and $\alpha_n''$. By
Proposition \ref{propc4}, the above inequality implies
\begin{equation*}
\|\ell_{\alpha'}-\ell_{\alpha''}\|_\infty \leq |\alpha'|_\infty
\|\ell_{\alpha'}-\ell_{\alpha''}\|_\infty
+|\alpha'-\alpha''|_\infty (\|\ell_{\alpha''}\|_\infty +M_{0,2p})
\end{equation*}
which gives (\ref{eqnc43}).
\end{proof}

An estimate on  $L^\infty$-error between the classical Lidstone
Interpolation Function and Lidstone FIF for a given data is now
found in the following corollary of Theorem 4.1:

\begin{cor}\label{thmc5a}
Let $\{(x_n,y_{n,2k});n=0,1,\ldots,N \text{ and }
k=0,1,\ldots,p\}$ be a given data and $\alpha \in
\Theta_{2p}$. Let $\ell_\alpha$ be the Lidstone FIF and $\phi$ be
the classical Lidstone Interpolation Function for the above data.
Then
\begin{equation}\label{eqnc46}
\|\ell_\alpha-\phi\|_\infty \leq \frac{|\alpha|_\infty}
{1-|\alpha|_\infty}(\|\phi\|_\infty+M_{0,2p})
\end{equation}
where,
$M_{0,2p}=\frac{2\rho\pi}{3}\sum_{l=0}^p\left(\frac{x_N-x_0}{\pi}\right)^{2l}$.
\end{cor}
\begin{proof}
Inequality (\ref{eqnc46}) is an immediate consequence of
Inequality (\ref{eqnc43}) with  $\alpha'=\alpha$ and $\alpha''=0$
, since $\ell_0=\phi$.
\end{proof}
\begin{rem}\label{remc3}
Since \ $\alpha \in \Theta_{2p}$ \ in Theorem 4.1 implies that $|\alpha_n|< a_n^{2p}$, the inequality $|\alpha|_\infty \leq(x_N-x_0)^{-2p} \|\Delta\|^{2p}$ holds, where $\|\Delta\|=max_k |x_k-x_{k-1}|$ is the norm of the partition $x_0<x_1<\cdots<x_N$ of the interval $[x_0,x_N]$. Therefore, it follows from Corollary \ref{thmc5a} that the convergence of Lidstone FIF to
Classical Lidstone Interpolation Function is of the order $\|\Delta\|^{2p}$.
\end{rem}
Using Corollary \ref{thmc5a}, the order of convergence of Lidstone
FIF to the data generating function is found in the following
theorem:
\begin{thm}\label{thmc5}
Let $x_0<x_1<\cdots<x_N$  be a uniform partition of the interval
$[x_0,x_N]$ and $\{(x_n,y_{n,2k});n=0,1,\ldots,N \text{ and }
k=0,1,\ldots,p\}$ be a data generated by the function $g\in
C^{2p}[x_0,x_N]$. Let
$\ell_\alpha$ be the Lidstone FIF for the data. Then,
\begin{equation*}\label{eqnc50}
\|g-\ell_\alpha\|_\infty =O(N^{-2p}).
\end{equation*}
\end{thm}
\begin{proof}
Let  $L^\Delta_{p_{+1}}g$ be the classical Lidstone Interpolation
Function for data $\{(x_n,y_{n,2k});n=0,1,\ldots,N \text{ and }
k=0,1,\ldots,p\}$. Inequality (\ref{eqnc16}) gives
\begin{equation*}\label{eqnc48}
\|L^\Delta_{p_{+1}}g\|_\infty \leq 2d_{2p,0} \ |\Delta\|^{2p}
\|g^{(2p)}\|_\infty+\|g\|_\infty
\end{equation*}
where, $d_{2p,0}=\frac{(-1)^pE_{2p}}{2^{2p}(2p)!}$, $E_{2p}$
being $2p^{th}$-Euler number. Therefore, it follows by Corollary \ref{thmc5a} that
\begin{equation}\label{eqnc49}
\|\ell_\alpha-L^\Delta_{p_{+1}}g\|_\infty \leq
\frac{|\alpha|_\infty}
{1-|\alpha|_\infty}(2d_{2p,0}\|\Delta\|^{2p}
\|g^{(2p)}\|_\infty+\|g\|_\infty+M_{0,2p})
\end{equation}
where,
$M_{0,2p}=\frac{2\rho\pi}{3}\sum_{l=0}^p\left(\frac{x_N-x_0}{\pi}\right)^{2l}$.
Using Inequalities (\ref{eqnc16}) and (\ref{eqnc49}), the above
inequality becomes
\begin{equation*}\label{eqnc52}
\|g-\ell_\alpha\|_\infty \leq \frac{1}
{1-|\alpha|_\infty}[2d_{2p,0}\|\Delta\|^{2p}
\|g^{(2p)}\|_\infty+|\alpha|_\infty(\|g\|_\infty+M_{0,2p})].
\end{equation*}
Since
$|\alpha|_\infty<\frac{1}{N^{2p}}=\frac{\|\Delta\|^{2p}}{(x_N-x_0)^{2p}}$ implies
$\frac{1}{1-|\alpha|_\infty} \leq \frac{N^{2p}}{N^{2p}-1}$, it follows from the above inequality that
\begin{equation*}\label{eqnc53}
\|g-\ell_\alpha\|_\infty \leq \frac{1}
{N^{2p}-1}[2d_{2p,0}(x_N-x_0)^{2p}
\|g^{(2p)}\|_\infty+\|g\|_\infty+M_{0,2p}].
\end{equation*}
so that
\begin{equation}\label{eqnc50b}
\|g-\ell_\alpha\|_\infty \leq \vartheta_0N^{-2p}.
\end{equation}
where
\begin{equation*}\label{eqnc50a}
\vartheta_0=\frac{1}{1-N^{-2p}}[2d_{2p,0}(x_N-x_0)^{2p}
\|g^{(2p)}\|_\infty+\|g\|_\infty+M_{0,2p}]
\end{equation*}
Since $\vartheta_0$ is bounded for $N>1$, the result follows from
Inequality (\ref{eqnc50b}).
\end{proof}

Our next result on continuous dependence of even order derivatives $\ell_{\alpha}^{(2k)}$ of Lidstone FIF  $\ell_{\alpha}$ on the scaling factor $\alpha$ is used to  find estimates on $L^\infty$-error between even order derivatives of the classical Lidstone Interpolation Function and corresponding derivatives of Lidstone FIF for a given data. These estimates are then employed to determine the order of $L^\infty$-error in approximation of even order derivatives of data generating function by corresponding derivatives of Lidstone FIF. \newpage

\begin{thm}\label{propc8}
Let $\ell_{\alpha'}$ and $\ell_{\alpha''}$ be the Lidstone FIFs
with scaling factors  $\alpha', \alpha'' \in \Theta_{2p}$ respectively, for a data $\{(x_n,y_{n,2k});n=0,1,\ldots,N \text{ and } k=0,1,\ldots,p\}$. Then, for $k=1,2,\ldots,p$,
\begin{equation}\label{eqnc57}
\|\ell_{\alpha'}^{(2k)}-\ell_{\alpha''}^{(2k)}\|_\infty \leq
\frac{|\alpha'-\alpha''|_\infty}{\mu^{2k}-|\alpha'|_\infty}
(\|\ell^{(2k)}_{\alpha''}\|_\infty +M_{2k,2p})
\end{equation}
where, $\mu=\min\{a_n:1\leq n \leq N\}$ and
$M_{2k,2p}=\frac{2\rho\pi}{3}\sum_{l=0}^{p-k}\left(\frac{x_N-x_0}{\pi}\right)^{2l}$.
\end{thm}
\begin{proof}
Since the $2k^{th}$ derivative $\ell_\alpha^{(2k)}$ of the
Lidstone FIF $\ell_\alpha$, is the fixed point of
Read-Bajraktarevi\'{c} Operator $T_\alpha^{(2k)}$ defined by
(\ref{eqnc30}), for $x\in [x_{n-1},x_n]$,
\begin{equation*}\label{eqnc26}
\ell_\alpha^{(2k)}(x)=\frac{\alpha_n\ell_\alpha^{(2k)}\circ
L_n^{-1}(x)+q_n^{(2k)}\circ L_n^{-1}x}{a_n^{2k}}.
\end{equation*}
Thus, by mean value theorem, for each $x\in [x_{n-1},x_n]$,
\begin{eqnarray*}
|\ell_{\alpha'}^{(2k)}(x)-\ell_{\alpha''}^{(2k)}(x)|
&\leq&\frac{1}{a_n^{2k}}|\alpha_n'\ell_{\alpha'}^{(2k)}(L_n^{-1}(x))-\alpha_n''\ell_{\alpha''}^{(2k)}(L_n^{-1}(x))|\\
&
&~~+\frac{1}{a_n^{2k}}|q_n^{(2k)}(\alpha_n',L_n^{-1}(x))-q_n^{(2k)}(\alpha_n'',L_n^{-1}(x))|.\\
&\leq &~\frac{|\alpha_n'|}{a_n^{2k}}~|\ell_{\alpha'}^{(2k)}(L_n^{-1}(x)) -\ell_{\alpha''}^{(2k)}(L_n^{-1}(x))|\\
& &~~+\frac{|\alpha_n'-\alpha_n''|}{a_n^{2k}}~|\ell_{\alpha''}^{(2k)}(L_n^{-1}(x))|\\
&
&~~+\frac{|\alpha_n'-\alpha_n''|}{a_n^{2k}}\left|\frac{\partial}{\partial\alpha_n}q_n^{(2k)}(\xi_n,L_n^{-1}(x))\right|
\end{eqnarray*}
 for some $\xi_n$ lying between $\alpha_n'$ and
$\alpha_n''$. Using Proposition \ref{propc4}, the above inequality implies
\begin{equation*}
\|\ell^{(2k)}_{\alpha'}-\ell^{(2k)}_{\alpha''}\| \leq
\frac{|\alpha'|_\infty}{\mu^{2k}}
\|\ell^{(2k)}_{\alpha'}-\ell^{(2k)}_{\alpha''}\|
+\frac{|\alpha'-\alpha''|_\infty}{\mu^{2k}}
(\|\ell^{(2k)}_{\alpha''}\|_\infty +M_{2k,2p})
\end{equation*}
which gives (\ref{eqnc57}).
\end{proof}

The estimates on $L^\infty$-error between even order derivatives of the classical Lidstone Interpolation Function and corresponding derivatives of Lidstone FIF for a given data are now given by the following corollary of Theorem 4.2: \newpage
\begin{cor}\label{thmc6a}
For a data $\{(x_n,y_{n,2k});n=0,1,\ldots,N \text{ and }
k=0,1,\ldots,p\}$ and scaling factor $\alpha \in \Theta_{2p}$, let $\ell_\alpha$ be the Lidstone FIF determined by Theorem \ref{thmc4} and $\phi$
be the classical Lidstone Interpolation Function. Then, for $k=1,2,\ldots,p$,
\begin{equation}\label{eqnc46a}
\|\ell_\alpha^{(2k)}-\phi^{(2k)}\|_\infty \leq
\frac{|\alpha|_\infty}
{\mu^{2k}-|\alpha|_\infty}(\|\phi^{(2k)}\|_\infty+M_{2k,2p})
\end{equation}
where, $\mu=\min\{a_n:1\leq n \leq N\}$ and
$M_{2k,2p}=\frac{2\rho\pi}{3}\sum_{l=0}^{p-k}\left(\frac{x_N-x_0}{\pi}\right)^{2l}$.
\end{cor}
\begin{proof}
Inequality (\ref{eqnc46a}) is an immediate consequence of
Inequality (\ref{eqnc57}) with  $\alpha'=\alpha$ and $\alpha''=0$
since $\ell_0=\phi$.
\end{proof}
Using Corollary \ref{thmc6a}, the order of $L^\infty$-error
between even order derivatives of data generating function and the corresponding derivatives  of its Lidstone FIF is found in the following theorem:
\begin{thm}\label{thmc6}
Let $x_0<x_1<\cdots<x_N$ be a uniform partition of $[x_0,x_N]$ and
$\{(x_n,y_{n,2k})$; $n=0,1,\ldots,N$ and $k=0,1,\ldots,p\}$ be a
data generated by a function $g\in C^{2p}[x_0,x_N]$. For $\alpha \in \Theta_{2p}$, let $\ell_\alpha$ be the Lidstone FIF for the above data determined by Theorem \ref{thmc4}. Then, for $k=1,2,\ldots,p$,
\begin{equation}\label{eqnc60}
\|g^{(2k)}-\ell^{(2k)}_\alpha\|_\infty =O(N^{-(2p-2k)}).
\end{equation}
\end{thm}
\begin{proof}
Inequalities (\ref{eqnc16}), (\ref{eqnc46a}) and the triangle
inequality give
\begin{eqnarray}\label{eqnc62}
\|g^{(2k)} - \ell^{(2k)}_\alpha\|_\infty &\leq
&~2d_{2p,2k}\|\Delta\|^{2p-2k} \|g^{(2p)}\|_\infty\nonumber\\
\nonumber & &~+\frac{|\alpha|_\infty} {\mu^{2k}
 -  |\alpha |_\infty}[2d_{2p,2k}\|\Delta\|^{2p-2k}
\|g^{(2p)}\|_\infty + \|g^{(2k)}\|_\infty + M_{2k,2p}]\\
\end{eqnarray}
where $M_{2k,2p}=\frac{2\rho\pi}{3}\sum_{l=0}^{p-k}
\left(\frac{x_N-x_0}{\pi}\right)^{2l}$ and $d_{2p,2k}$ is given by
(\ref{eqnc15}). As $|\alpha|_\infty<\frac{1}{N^{2p}}=
\frac{\|\Delta\|^{2p}}{(x_N-x_0)^{2p}}$, there exits $s>0$ such
that $|\alpha|_\infty=\frac{1}{N^{2p+s}}$ and therefore
$\frac{\alpha|_\infty}{\mu^{2k}-|\alpha|_\infty}<
\frac{1}{N^{2p+s-2k}-1}$. Consequently, Inequality (\ref{eqnc62})
becomes
\begin{eqnarray*}
\|g^{(2k)}-\ell^{(2k)}_\alpha\|_\infty &\leq&~
2d_{2p,2k}\|\Delta\|^{2p-2k}
\|g^{(2p)}\|_\infty\\
& &~+\frac{2d_{2p,2k}\|\Delta\|^{2p-2k}
\|g^{(2p)}\|_\infty+\|g^{(2k)}\|_\infty+M_{2k,2p}}
{N^{2p+s-2k}-1}.
\end{eqnarray*}
Since $\|\Delta\|=\frac{x_N-x_0}{N}$, the above inequality gives
\begin{equation}\label{eqnc64}
\|g^{(2k)}-\ell^{(2k)}_\alpha\|_\infty  \leq
\vartheta_kN^{-(2p-2k)}
\end{equation}
where
\begin{equation*}\label{eqnc64a}
\vartheta_k= \frac{1}
{1-N^{-(2p+s-2k)}}\left[2d_{2p,2k}(x_N-x_0)^{2p-2k}
\|g^{(2p)}\|_\infty+\frac{\|g^{(2k)}\|_\infty+M_{2k,2p}}{N^{2p+s-2k}}\right].
\end{equation*}
The Estimate (\ref{eqnc60}) follows from Inequality (\ref{eqnc64}), since it is easily seen that $\vartheta_k$ is bounded for $N>1$.
\end{proof}

\section[Computational Modeling of  Lidstone FIF and Lidstone CFIF]
{Computational Modeling of  Lidstone FIF and its Derivatives}\label{examplesc}

The results obtained in Section \ref{consLFIF}  are illustrated in this section
by computationally constructing a $C^4$-Lidstone FIF for a data $\{(x_n,y_{n,k}); n=0,1,\ldots,10$, $k=0,1,2\}$ generated by the function {\Large $\frac{\sin x}{x}$} in the
interval $[5,25]$. A generated data along with chosen scaling factors are given in Table \ref{tableex1c1}. The random \ iteration \ algorithm \ with $20000$ iterations is used to simulate the $C^4$-Lidstone FIF and its derivatives using the construction method developed in Theorem \ref{thmc4}. The resulting plots of Lidstone FIF and its second and fourth order derivatives are given in Figure 1.
\newpage
\begin{table}[ht]
\begin{center}
\begin{tabular}{|c|r|r|r|r|r|}
\hline
 $n$  & $x_n$ & $y_{n,0}$ & $y_{n,2}$  & $y_{n,4}$ & $\alpha_n$ \\
\hline
 $0$  & $ 5$ &$-0.1918$ &$ 0.1991$  & $-0.0508$ & -- \\

 $1$  &$ 7$ &$ 0.0939$ &$-0.0593$  &$ 0.1409$ &$0.3162\times 10^{-4}$ \\

 $2$  &$ 9$ &$ 0.0458$ &$-0.0672$  &$-0.0092$ &$0.3981\times 10^{-4}$ \\

 $3$  &$11$ &$-0.0909$ &$ 0.0895$  &$-0.0819$ &$0.1585\times 10^{-4}$ \\

 $4$  &$13$ &$ 0.0323$ &$-0.0212$  &$ 0.0523$ &$0.1995\times 10^{-4}$ \\

 $5$  &$15$ &$ 0.0434$ &$-0.0497$  &$ 0.0272$ &$0.2512\times 10^{-4}$ \\

 $6$  &$17$ &$-0.0566$ &$ 0.0543$  &$-0.0581$ &$0.3981\times 10^{-4}$ \\

 $7$  &$19$ &$ 0.0079$ &$-0.0024$  &$ 0.0188$ &$0.3802\times 10^{-4}$ \\

 $8$  &$21$ &$ 0.0398$ &$-0.0421$  &$ 0.0337$ &$0.2512\times 10^{-4}$ \\

 $9$  &$23$ &$-0.0368$ &$ 0.0346$  &$-0.0400$ &$0.3090\times 10^{-4}$ \\

$10$  &$25$ &$-0.0053$ &$ 0.0084$  &$ 0.0012$ &$0.3162\times 10^{-4}$ \\
\hline
\end{tabular}
\caption{\small {A data generated by {\large$\frac{\sin x}{x}$} and chosen scaling factors $\alpha_n$}\label{tableex1c1}}
\end{center}
\end{table}
\newpage
\begin{figure}[!hbp]
\centering
\subfigure[$c^4$-Lidstone FIF]{\epsfig{file=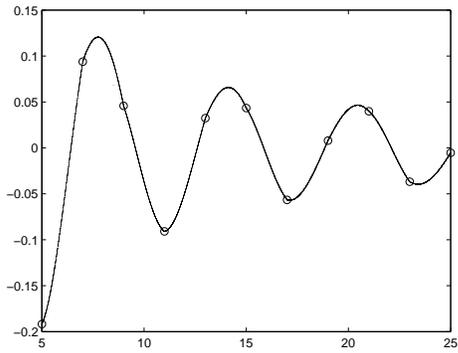, width=7cm}\label{fig0}}\hspace{2cm}
\subfigure[Second derivative of $c^4$-Lidstone FIF]{\epsfig{file=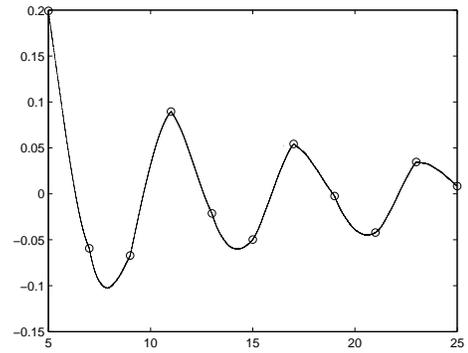, width=7cm}}\\
\subfigure[Fourth derivative of $c^4$-Lidstone FIF
]{\epsfig{file=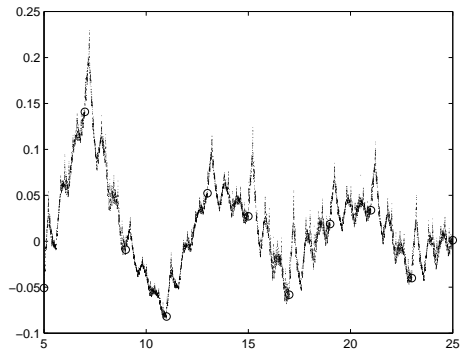, width=7cm}\label{m0}}\hspace{2cm}
\caption{\small{$c^4$-Lidstone FIF for given data and chosen scaling
factors (c.f. Table \ref{tableex1c1})}}
\end{figure}

\newpage

\section{Conclusions}
The \ classical \ Lidstone \ Interpolation is \ extended \ in \ the \ \ present \ paper as Lidstone FIF to simulate a given data $\{(x_n,y_{n,2k}); \ n=0,1,\ldots,N \ \text{and} \ k=0,1,\ldots,p\}$  with $N,p\in\mathbb{N}$. The existence of Lidstone FIF is established in our work and a computational method for its construction is developed. The constructed Lidstone FIF is a $C^{2p}[x_0,x_N]$ fractal function $\ell_\alpha$ satisfying
$\ell_\alpha^{(2k)}(x_n)=y_{n,2k}$, $n=0,1,\ldots,N$,\ $k=0,1,\ldots,p$. Our error estimates establish that the order of $L^\infty$-error in approximation of a data generating function in $C^{2p}[x_0,x_N]$  by Lidstone FIF is of the order $N^{-2p}$, while $L^\infty$-error in approximation of $2k$-order derivative of the data generating function by corresponding order derivative of Lidstone FIF is of the order $N^{-(2p-2k)}$. The results found in the present work are illustrated for computational constructions of a Lidstone FIF and its derivatives with an example of a data generating function.

\end{document}